\documentclass[12pt]{amsart}
\usepackage{amssymb,amsmath,amsthm,comment}
\usepackage{pdfpages,graphicx} 
\usepackage[section]{placeins} 
\usepackage{wrapfig}
\usepackage{float}
\usepackage{lineno}
\usepackage{youngtab}

\usepackage{setspace}
\newcommand{\shrinkmargins}[1]{
  \addtolength{\textheight}{#1\topmargin}
  \addtolength{\textheight}{#1\topmargin}
  \addtolength{\textwidth}{#1\oddsidemargin}
  \addtolength{\textwidth}{#1\evensidemargin}
  \addtolength{\topmargin}{-#1\topmargin}
  \addtolength{\oddsidemargin}{-#1\oddsidemargin}
  \addtolength{\evensidemargin}{-#1\evensidemargin}
  }
\shrinkmargins{0.5}

\newtheorem{theorem}{Theorem}
\newtheorem{lemma}[theorem]{Lemma}

\newtheorem{corollary}[theorem]{Corollary}
\newtheorem*{theorem*}{Theorem}
{Claim}

\theoremstyle{definition}

\theoremstyle{remark}

\newtheorem*{remark}{Remark}
\newtheorem*{remarks}{{\bf Remarks}}

\numberwithin{theorem}{section} \numberwithin{equation}{section}

\usepackage{caption}
\usepackage{multirow}

\def\func#1{\mathop{\rm #1}}%

\begin{document}
\title[Bessenrodt--Ono inequalities for $\ell $-tuples]{Bessenrodt--Ono inequalities for $\ell $-tuples of pairwise commuting permutations}
\author{Abdelmalek Abdesselam}
\address{Department of Mathematics, P. O. Box 400137, University of Virginia, Charlottesville, VA 22904-4137, USA}
\email{malek@virginia.edu}
\author{Bernhard Heim}
\address{Department of Mathematics and Computer Science\\Division of Mathematics\\University of Cologne\\ Weyertal 86-90 \\ 50931 Cologne \\Germany}
\address{Lehrstuhl A f\"{u}r Mathematik, RWTH Aachen University, 52056 Aachen, Germany}
\email{bheim@uni-koeln.de}
\author{Markus Neuhauser}
\address{Kutaisi International University, 5/7, Youth Avenue,  Kutaisi, 4600 Georgia}
\address{Lehrstuhl A f\"{u}r Mathematik, RWTH Aachen University, 52056 Aachen, Germany}
\email{markus.neuhauser@kiu.edu.ge}
\begin{abstract}
Let $S_n$ denote the symmetric group. We consider
\begin{equation*}
N_{\ell}(n) := \frac{\left\vert \func{Hom}\left( \mathbb{Z}^{\ell},S_n\right) \right\vert}{n!}
\end{equation*}
which
also counts the number
of $\ell$-tuples $\pi=\left( \pi_1, \ldots, \pi_{\ell}\right) \in S_n^{\ell}$ with
$\pi_i \pi_j = \pi_j \pi_i$ for $1 \leq i,j \leq \ell$ scaled by $n!$.
A recursion formula, generating function, and Euler product have been discovered
by Dey, Wohlfahrt, Bryman and Fulman, and White.
Let $a,b, \ell \geq 2$.
It is known by Bringman, Franke, and Heim, that the Bessenrodt--Ono inequality
\begin{equation*}
\Delta_{a,b}^{\ell}:= N_{\ell}(a) \, N_{\ell}(b) - N_{\ell}(a+b) >0
\end{equation*}
is valid for $a,b \gg 1$ and by Bessenrodt and Ono that
it is valid for $\ell =2$ and $a+b >9$.
In this paper we prove that for each pair $(a,b)$ the sign  of 
$\{\Delta_{a,b}^{\ell} \}_{\ell}$ is getting stable. 
In each case we provide an explicit bound.



The numbers $N_{\ell
}\left( n\right) $ had been identified by Bryan and Fulman as the $n$-th orbifold characteristics, generalizing work
by Macdonald and  Hirzebruch--H\"{o}fer concerning the ordinary
and string-theoretic Euler characteristics of symmetric products, where
$N_2(n)=p(n) $ represents the partition function.
\end{abstract}
\maketitle
\newpage
\section{Introduction and Main Results}
In this paper we consider the Bessenrodt--Ono
inequality \cite{BO16} for commuting $\ell$-tuples in the symmetric group $S_n$.
These sequences had been investigated in number theory, combinatorics, algebra, and
physics. We refer to the work by
Hirzebruch and H\"ofer \cite{HH90}, Bryman and Fulman \cite{BF98},
Liskovets and Mednykh \cite{LM09}, 
Britnell \cite{Br13}, White, \cite{Wh13}, Abdesselam, Brunialti, Doan, and Velie \cite{ABDV24}, Bringman, Franke, and Heim \cite{BFH24}, and
Abdesselam, Heim, and Neuhauser \cite{AHN24}.

We consider the double sequence $\{N_{\ell}(n)\}$ of positive integers
provided by properties of the symmetric group $S_n$ on $n$ elements.
Let $\ell,n  \geq 1$ and 
\begin{equation*}
C_{\ell,n}:= \Big\{ \pi=(\pi_1, \ldots, \pi_{\ell}) \in S_n^{\ell}\, : \,
\pi_j \pi_k = \pi_k \pi_j \text{ for } 1 \leq j,k \leq \ell \Big\}.
\end{equation*}
Every $\pi \in C_{\ell,n}$ corresponds to a group homomorphism from
$\mathbb{Z}^{\ell} $ to $S_n$. We define
\begin{equation*}
N_{\ell}(n):= \frac{\vert C_{\ell,n} \vert}{n!} 
=
\frac{\left\vert 
\func{Hom}\left( \mathbb{Z}^{\ell},S_n\right) \right\vert
}{n!}  \in \mathbb{N}, \text{ with }  N_{\ell}(0):=1.
\end{equation*}
Dey \cite{De65} and Wohlfahrt \cite{Wo77} provide an interesting general recursive approach
to determine the number of homorphisms $n! \, N_G(n):=  \left\vert \func{Hom}\left( G,S_n\right) \right\vert$ 
of a finitely generated group $G$ into $S_n$. Let $g_G(n)$ denote the
number of subgroups of $G$ of index $n$.
Note, there are extensive activities
studying the growth of these numbers in characterizing the group $G$ itself
by Lubotzky and Segal \cite{LS03}:
\begin{equation} 
\label{eq:Wohlfahrt}
N_G(n) = \frac{1}{n} \sum_{k=1}^n g_G(k) \, N_G(n-k), 
\end{equation}
with initial value $N_G(0):=1$. 
Let $G=\mathbb{Z}^{\ell}$
and $g_{\ell}(n):= g_G(n)$. Then $g_1(n)=1$ and $g_2(n) = \sigma(n) = \sum_{d
\mid n} d$.
Therefore, we recover in the case $\ell=2$ from (\ref{eq:Wohlfahrt})  Euler's
well known recursion formula for the partition numbers $p(n)$ \cite{An98, On04}.
The case $\ell =3$ reveals a connection to topology 
as highlighted by Britnell (see the introduction of \cite{Br13}) due to work by
Liskovets and Medynkh \cite{LM09}. They discovered that  
$N_{3}(n)$ counts the number
of non-equivalent $n$-sheeted coverings of a torus. Further, 
the numbers $N_{\ell}(n)$ had been identified by Bryan and Fulman \cite{BF98}
as the $n$-th orbifold characteristics. This generalizes work by
Macdonald \cite{Ma62} and  Hirzebruch and H\"ofer \cite{HH90} concerning the ordinary
and string-theoretic Euler characteristics of symmetric products.
To illustrate the growth of $\{N_{\ell}(n)\}_{\ell,n}$ we have recorded the first values for $0 \leq n \leq 10$ and  $2 \leq \ell \leq 4$ in Table \ref{tab:Nell}. 
Note that $N_{1}(n)=1$.
%
\begin{table}[H]
\[
\begin{array}{crrrrrrrrrrr}
\hline
n & 0 & 1 & 2 & 3 & 4 & 5 & 6 & 7 & 8 & 9 & 10 \\ \hline \hline
N_2(n) & 1 & 1 & 2 &
3 &
5 &
7 &
11 &
15 &
22 & 30 & 42 \\
N_3(n) & 1 & 1 &4&8&21&39&92&170&360&667&1316 \\
N_4(n) &1 & 1 &8&21&84&206&717&1810&5462&13859&38497 \\
\hline
\end{array}
\]
\caption{\label{tab:Nell}
Values for $ 0 \leq n \leq 10$.}
\end{table}
In this paper we investigate
the satisfiability
of the Bessenrodt--Ono inequality $\Delta_{a,b}^{\ell}>0$, where $a,b,\ell \geq 2$:
\begin{equation*} \label{def:BO}
\Delta_{a,b}^{\ell}:= N_{\ell}(a) \, N_{\ell}(b) 
- N_{\ell}(a +b).
\end{equation*}
Bessenrodt and Ono \cite{BO16} proved the first result in this direction.
Utilizing the well-known Lehmer bound for partition numbers, they 
proved for $p(n) = N_{2}(n)$:
\begin{theorem}[Bessenrodt, Ono] \label{th:BO}
Let $a,b \geq 2$ and $a+b >9$. Then $\Delta_{a,b}^2 >0$. Let
\begin{eqnarray*} 
E_{2}&:=& \{(2,6), (6,2), (2,7), (7,2), (3,4),(4,3)\},\\
F_{2}&:=& \{(2,2), (2,3), (3,2), (2,4), (4,2), (2,5), (5,2), (3,3), (3,5), (5,3)\}.
\end{eqnarray*}
Then $\Delta_{a,b}^2=0$ if
and only if $(a,b) \in E_2$ and $\Delta_{a,b}^2<0$ if
and only if $(a,b) \in F_2$.
\end{theorem}
Recently, it has been shown for every $\ell \geq 2$ (\cite{BFH24} by Bringmann, Franke, and Heim), that $\Delta_{a,b}^{\ell}>0$ for $a,b \gg 1$.
\begin{remarks}
Let $E_{\ell}$ and $F_{\ell}$ denote the sets of all $(a,b)$, where
$\Delta_{a,b}^{\ell}=0$ and $\Delta_{a,b}^{\ell}<0$, resp.
Numerical investigations suggest that $E_{\ell} = \emptyset $ for $ \ell \geq 3$.
We have recorded the sets $F_{\ell}^{+}$ of all $\left( a,b\right) $
with $\Delta _{a,b}^{\ell }<0$
for $2 \leq \ell \leq 10$ and
$2 \leq b\leq a \leq 1000$
in Table \ref{tab:exceptions}. 
It also seems that the sets $F_{\ell}$ are finite. This does not follow from \cite{BFH24}.
\end{remarks}
\begin{table}[H]
\[
\begin{array}{rl}
\hline
\ell &F_{\ell }^{+}\\ \hline \hline
2&\left\{ \left( 2,2\right) ,\left( 2,3\right) ,\left( 2,4\right) ,\left( 2,5\right) ,\left( 3,3\right) ,\left( 3,5\right) \right\} \\
3&\left\{ \left( 2,2\right) ,\left( 2,3\right) ,\left( 2,4\right) ,\left( 2,5\right) ,\left( 3,3\right) ,\left( 3,4\right) ,\left( 3,5\right) \right\} \\
4&\left\{ \left( 2,2\right) ,\left( 2,3\right) ,\left( 2,4\right) ,\left( 2,5\right) ,\left( 3,3\right) ,\left( 3,4\right) ,\left( 3,5\right) ,\left( 3,7\right) \right\} \\
5&\left\{ \left( 2,2\right) ,\left( 2,3\right) ,\left( 2,4\right) ,\left( 2,5\right) ,\left( 3,3\right) ,\left( 3,4\right) ,\left( 3,5\right) ,\left( 3,7\right) \right\} \\
6&\left\{ \left( 2,2\right) ,\left( 2,3\right) ,\left( 2,4\right) ,\left( 2,5\right) ,\left( 3,3\right) ,\left( 3,4\right) ,\left( 3,5\right) ,\left( 3,7\right) \right\} \\
7&\left\{ \left( 2,2\right) ,\left( 2,3\right) ,\left( 2,4\right) ,\left( 2,5\right) ,\left( 3,3\right) ,\left( 3,4\right) ,\left( 3,5\right) ,\left( 3,7\right) \right\} \\
8&\left\{ \left( 2,2\right) ,\left( 2,3\right) ,\left( 2,4\right) ,\left( 2,5\right) ,\left( 3,3\right) ,\left( 3,4\right) ,\left( 3,5\right) \right\} \\
9&\left\{ \left( 2,2\right) ,\left( 2,3\right) ,\left( 2,4\right) ,\left( 2,5\right) ,\left( 3,3\right) ,\left( 3,4\right) ,\left( 3,5\right) \right\} \\
10&\left\{ \left( 2,2\right) ,\left( 2,3\right) ,\left( 2,4\right) ,\left( 2,5\right) ,\left( 3,3\right) ,\left( 3,4\right) ,\left( 3,5\right) \right\} \\
 \hline
\end{array}
\]
\caption{\label{tab:exceptions} Strict exceptions $F_{\ell }^{+}$ to given $\ell $ for $2 \leq a
\leq b
\leq 1000$
}
\end{table}
\begin{theorem} \label{th:unbounded}
Let $L(\ell)$ be the smallest number such that $\Delta_{a,b}^{\ell}>0$ for all \\ $a,b \geq L(\ell)$. Then $\{L(\ell)\}_{\ell}$ is unbounded.
\end{theorem}
In \cite{BFH24} an asymptotic formula for $N_{\ell}(n)$ for each $\ell$ was given. In this paper we fix $n$. Then we deduce from (\cite{AHN24}, Lemma 3.1)
an asymptotic formula in the $\ell$-aspect. This leads to
the following for $n>1$
\begin{equation} \label{eq:asymp}
N_{\ell }\left( n\right) \sim B
_{\ell }\left( n\right) =\left\{
\begin{array}{cl}
\frac{1}{2^{n/3}\left( \frac{n}{3}\right) !}3^{\left( \ell -1\right) n/3},&n\equiv 0\mod 3,\\
\frac{7}{6}\frac{1}{2^{\frac{n-4}{3}}
\left( \frac{n-4}{3}\right) !}\left( 4\cdot 3^{\frac{n-4}{3}}\right) ^{\ell-1},&n\equiv 1\mod 3,\\
\frac{1}{2^{\frac{n-2}{3}}\left( \frac{n-2}{3}\right) !}\left( 2\cdot 3^{\frac{n-2}{3}}\right) ^{\ell -1},&n\equiv 2\mod 3.
\end{array}
\right. 
\end{equation}
Table \ref{100} illustrates the approximation provided by (\ref{eq:asymp}) for $n=100$.
\begin{table}[H]
\[
\begin{array}{rrrr}
\hline
\ell &N_{\ell }\left( 100\right) &B_{\ell }\left( 100\right) &\frac{B_{\ell }\left( 100\right) }{N_{\ell }\left( 100\right) }\\ \hline \hline
2&1.905692920\cdot 10^{8}&7.651654004\cdot 10^{-30}&0.000000000\\
3&2.173101780\cdot 10^{19}&5.671467739\cdot 10^{-14}&0.000000000\\
4&4.596301073\cdot 10^{31}&4.203737688\cdot 10^{2}&0.000000000\\
5&5.233985605\cdot 10^{44}&3.115844322\cdot 10^{18}&0.000000000\\
6&1.866914831\cdot 10^{58}&2.309488974\cdot 10^{34}&0.000000000\\
7&1.544852719\cdot 10^{72}&1.711811878\cdot 10^{50}&0.000000000\\
8&2.441267698\cdot 10^{86}&1.268808787\cdot 10^{66}&0.000000000\\
9&6.417286329\cdot 10^{100}&9.404513196\cdot 10^{81}&0.000000000\\
10&2.532012849\cdot 10^{115}&6.970701127\cdot 10^{97}&0.000000000\\
20&1.753717670\cdot 10^{266}&3.488814560\cdot 10^{256}&0.000000000\\
30&3.930913320\cdot 10^{420}&1.746140999\cdot 10^{415}&0.000004442\\
40&4.019856937\cdot 10^{576}&8.739382209\cdot 10^{573}&0.002174053\\
50&6.226773355\cdot 10^{733}&4.374034023\cdot 10^{732}&0.070245595\\
60&5.888107470\cdot 10^{891}&2.189190629\cdot 10^{891}&0.371798687\\
70&1.523836227\cdot 10^{1050}&1.095683203\cdot 10^{1050}&0.719029502\\
80&6.087303965\cdot 10^{1208}&5.483860866\cdot 10^{1208}&0.900868578\\
90&2.835124887\cdot 10^{1367}&2.744655564\cdot 10^{1367}&0.968089828\\
100&1.387516910\cdot 10^{1526}&1.373691701\cdot 10^{1526}&0.990036007\\ \hline
\end{array}
\]
\caption{\label{100} Comparison of $N_{\ell }\left( n\right) $ and $B_{\ell }\left( n\right) $ for $n=100$ and $\ell =2,3,\ldots ,10,20,\ldots ,100$}
\end{table}


This leads to the following result capturing what happens for each pair $(a,b)$ for almost all $\ell$.
\begin{theorem}\label{th:almost}
Let $a,b, \ell \geq 2$ and 
$\Delta_{a,b}^{\ell}= N_{\ell}(a) N_{\ell}(b) - N_{\ell}(a+b)$.

\begin{itemize}
\item
Let $a \equiv 0 \pmod{3}$. \\ Then $\Delta_{a,b}^{\ell} >0$ for almost all $\ell$ if
$(a,b) \neq \left( 3,2\right) , \left( 3,4\right) ,\left( 2,3\right) ,\left( 4,3\right) $
and $\Delta_{a,b}^{\ell} <0 $ for almost all $\ell$ otherwise.
\item
Let $a \equiv 1,2 \pmod{3}$ and $b \equiv 1,2 \pmod{3}$
but not both $\equiv 2\pmod{3}$. \\Then $\Delta_{a,b}^{\ell} <0$ for almost all $\ell$.
\item
Let $a \equiv 2 \pmod{3}$ and $b \equiv 2 \pmod{3}$.\\
Then $\Delta_{a,b}^{\ell} >0$ for almost all $\ell$ if $a,b >2$ and $\Delta_{a,b}^{\ell} <0$ for almost all $\ell$ otherwise ($a=2$ or $b=2$). 
\end{itemize}
\end{theorem}
With some extra effort and extension
of (\ref{eq:asymp}), we can make the bounds effective.
Let $(a,b)$ be fixed, where $ 2 \leq a
\leq b \leq 10$.
Then we have recorded for every $\ell \geq 2$
in Table {\ref{tab:explicit} if
\begin{equation*}
\Delta_{a,b}^{\ell}>0 
\,\,\text{ or } \,\,\Delta_{a,b}^{\ell}<0
\end{equation*}
Note that $\Delta_{a,b}^{1}=0$ for all $a,b$.
$+
_m$ indicates that 
$\Delta_{a,b}^{\ell}>0$ for $\ell\ge m$ and $\Delta_{a,b}^{\ell}<0$ for $2\le \ell \le m-1$. Likewise,
$-_{m}$ indicates $\Delta_{a,b}^{\ell}<0$ for $\ell\ge m$ and $\Delta_{a,b}^{\ell}>0$ for $2\le m-1$.
Special cases: $+_{3}^{\ast }$ indicates,
$\Delta _{2,6}^{2}=0$, and
$\Delta _{2,6}^{\ell }>0$ for $\ell \geq 3$
and $-_{19}^{*
}$ indicates $\Delta_{2,7}^2=0$, $\Delta_{2,7}^{\ell}>0$ for $3 \leq \ell \leq 18$, and $\Delta_{2,7}^{\ell} <0$ for $\ell \geq 19$.
Moreover, $-_{3}^{\ast }$ indicates
$\Delta _{3,4}^{2}=0$ and
$\Delta _{3,4}^{\ell }<0$ for $\ell \geq 3
$
and $+_8^{*
}$ indicates 
$\Delta_{3,7}^{\ell}<0$ for 
$4 \leq \ell \leq 7$ and
$\Delta_{3,7}^{\ell}>0$ otherwise.

\begin{table}[H]
\[
\begin{array}{c|c|c|c|c|c|c|c|c|c|}
\hline
a\backslash b & 2 & 3 & 4 &5&6&7&8&9&10\\
\hline 
2 & -_2&  -_2& -_2&-_{2}&+_{3}^{\ast }&-_{19}^{*
}&-_{23}&+_{2}&-_{26}\\ \hline
3 &  &  +_{18}& -_{3}^{\ast } &+_{15}&+_{2}&+_{8}^{*
}&+_{2}&+_{2}&+_{2}\\ \hline
4 & & & -_{16}&-_{19}&+_{2}&-_{26}&-_{29}&+_{2}&-_{33}\\ \hline
5&&&&+_{2}&+_{2}&-_{38}&+_{2}&+_{2}&-_{47}\\ \hline
6&&&&&+_{2}&+_{2}&+_{2}&+_{2}&+_{2}\\ \hline
7&&&&&&-_{41}&-_{47}&+_{2}&-_{49}\\ \hline
8&&&&&&&+_{2}&+_{2}&-_{59}\\ \hline
9&&&&&&&&+_{2}&+_{2}\\ \hline
10&&&&&&&&&-_{59}\\ \hline
\end{array}
\]
\caption{\label{tab:explicit}$\pm_{m}$
positive or 
negative value of $\Delta_{a,b}^{\ell }$ for $\ell \geq m$.}
\end{table}

For example the exception 
$(2,2)$ 
in Bessenrodt and Ono's work related to 
$p(n) = N_{2}(n)$ is an exception for all $\ell \geq 2$ 
since $\Delta_{2,2}^{\ell} <0$ for all $\ell \geq 2$. But
$(3,3)$ for $\ell =2$
is only an exception until $\ell =17$: 
$\Delta_{3,3}^{\ell} <0$ for $2 \leq \ell \leq 17$ and
$\Delta_{3,3}^{\ell} >0$ for $\ell \geq 18$.

\begin{theorem}\label{th:Main}
Let $2 \leq a\leq 
b \leq 10$. Then Table \ref{tab:explicit} records all $\ell \geq m$ for
which $\Delta_{a,b}^{\ell}<0$ and $\Delta_{a,b}^{\ell}>0$.
\end{theorem}
We split our general result into two parts.

\begin{theorem}\label{th:explicit 1}
We 
consider $b=2$ or $b=4$.
\begin{enumerate}
\item  Let $b=2$.

\begin{enumerate}
\item  Let $a >9 $ and $a \equiv 0 \pmod{3}$. Then $\Delta_{a,b}^{\ell} >0$ for 
 $$\ell \geq 1+\log _{9/8}\left( \frac{72}{17}\right)
+\frac{a-6}{3}\log _{9/8}\left( \frac{2}{3}\left( a-6\right)
\right) +\left( a+2\right) \log _{9/8}\left( 2\right). $$

\item  Let $a >5$ and $a \equiv 2 \pmod{3}$. Then $\Delta_{a,b}^{\ell}<0$ for 
 \begin{equation*}
 \ell \geq 1+\frac{a-2}{3}\log _{9/8}\left(
\frac{2}{3}\left( a-2\right)
\right) +a\log _{9/8}\left( a\right) .
 \end{equation*}

\end{enumerate}

\item  Let $b=4$. Let $ a>9$ and $a \equiv 0 \pmod{3}$.
Then $\Delta_{a,b}^{\ell} <0$ for
 \begin{equation*}
\ell \geq 
 1+\log _{9/8}\left( \frac{360}{227}\right)
+\frac{a-6}{3}\log _{9/8}\left( \frac{2}{3}\left( a-6\right)
\right) +\left( a+4\right) \log _{9/8}\left( 2\right) .
 \end{equation*}
\end{enumerate}
\end{theorem}

\begin{theorem} 
\label{th:explicit 2}
Let $b \geq 2$.

\begin{enumerate}
\item  Let $b \equiv a \equiv 0 \pmod{3}$. Then $\Delta_{a,b}^{\ell} >0$ for 
\begin{equation*}
\ell >1+\frac{a+b}{3}\log _{9/8}\left( 4\frac{a+b}{3}\right) .
\end{equation*}

\item  Let $b \equiv 1 \pmod{3}$.

\begin{enumerate}
\item  Let $b>4$ and $a\equiv 0 \pmod{3}$. Then $\Delta_{a,b}^{\ell} >0$ for
\begin{equation*}
\ell \geq \log _{9/8}\left( \frac{7}{6}\right) +\frac{a+b-2}{3}\log _{9/8}\left( 4\right) +\frac{a+b-4}{3}\log _{9/8}\left( \frac{a+b-4}{3}\right). 
\end{equation*}

\item  Let $b>1$ and $ a \equiv 1 \pmod{3}$. Then $\Delta_{a,b}^{\ell} <0$ for 
\begin{equation*}
\ell \geq
-6+9 (a+b)+\frac{(a+b)/2+2}{3}\log _{9/8}\left( \frac{(a+b)-2}{3}\right).
\end{equation*}
\end{enumerate}

\item  Let $b \equiv 2 \pmod{3}$.

\begin{enumerate}
\item  Let $b>2$ and $a \equiv 0 \pmod{3}$. Then $\Delta_{a,b}^{\ell} >0$ for
\begin{equation*}
\ell \geq 1+\frac{a+b-1}{3}\log _{9/8}\left( 4\right) +\frac{a+b-2}{3}\log _{9/8}\left( \frac{a+b-2}{3}\right).
\end{equation*}

\item  Let $a \equiv 1 \pmod{3}$ and $a >1$. Then $\Delta_{a,b}^{\ell} <0$ for
\begin{equation*}
\ell \geq 5+9
\left( a+b\right) +
\frac{(a+b)/2+4}{3}
\log _{9/8}\left( \frac{a+b}{3}\right).
\end{equation*}

\item  Let $a,b >2$ and $a \equiv 2 \pmod{3}$. Then
$\Delta_{a,b}^{\ell}>0$ for 
\begin{equation*}
\ell \geq
1+\log _{9/8}\left( \frac{7}{6}\right) +\left( \frac{a+b-4}{3}+a+b\right) \log _{9/8}\left( 2\right) .
\end{equation*}
\end{enumerate}
\end{enumerate}
\end{theorem}
\section{Preliminaries}
Let $\ell \geq 1$ and $g_{\ell}(n)$ denote the amount of subgroups of
$\mathbb{Z}^{\ell}$ of index $n$. Then Dey and Wohlfahrt \cite{De65, Wo77} have obtained (\ref{eq:Wohlfahrt}) that
\begin{equation*}
N_{\ell}(n) = \frac{1}{n} \sum_{k=1}^n g_{\ell}(k) \, N_{\ell}(n-k), \, \, (N_{\ell}(0):=1).
\end{equation*}
Therefore, we have (see also \cite{BF98, ABDV24} for a polynomial version):
	\begin{equation*} \label{eq:product}
		\sum_{n=0}^{\infty} N_{\ell}(n)\, t^n  =
		\prod_{n=1}^{\infty} \left( 1-t^n\right)^{-  \, g_{\ell-1}(n)} =
		\exp \left( \, \sum_{n=1}^{\infty} g_{\ell}(n)  \frac{t^n}{n}\right),
	\end{equation*}
	where $g_0(n)=1$ for $n=1$ and $0$ otherwise.
This leads to	
\begin{equation}\label{explicit}
N_{\ell}(n) =
\sum _{1 \leq k\leq n}\sum _{\substack{m_{1},\ldots ,m_{k}\geq 1 \\ m_{1}+\ldots +m_{k}=n}}
\frac{1}{k!} \, \frac{g _{\ell
}\left( m_{1}\right) \cdots g _{\ell
}\left( m_{k}\right) }{m_{1}\cdots m_{k}}.
\end{equation}
The arithmetic function $g_{\ell}(n)$ has the following properties
(\cite{BF98, LS03, ABDV24, AHN24}).
Let $\ell \geq 1$. Then $g_{\ell}(n)$ is multiplicative and satisfies
for $p$ prime and $m \in \mathbb{N}$:

\begin{equation} \label{local}
g_{\ell}(p^m) = \frac{ \left(p^{\ell}-1 \right) \cdots
\left(p^{\ell + m -1}-1 \right)}{\left( p -1\right) \cdots
\left(p^m -1 \right)}.
\end{equation}
Moreover, we have the recursion formula
\begin{equation*}
g_{\ell}(n)= \sum_{d
\mid n} d \, g_{\ell -1}(d).
\end{equation*}
Further, we will deduce from (\ref{explicit}): 
\begin{lemma}
\label{klein}For $\ell \geq 1$ we have
\begin{eqnarray*}
N_{\ell }\left( 2\right) &=&2^{\ell -1}
,\\
N_{\ell }\left( 3\right) &=&\frac{1}{2}3^{\ell -1}+2^{\ell -1}-\frac{1}{2},\\
N_{\ell }\left( 4\right) &=&\frac{7}{6}4^{\ell -1}+\frac{1}{2}3^{\ell -1}-\frac{1}{2}2^{\ell -1}-\frac{1}{6
}
,\\
N_{\ell }\left( 5\right) &=&\frac{1}{2}6^{\ell -1}+\frac{1}{4}5^{\ell -1}+\frac{7}{6}4^{\ell -1}-2^{\ell -1}+\frac{1}{12}
,\\
N_{\ell
}\left( 7\right) &=&\frac{7}{12}12^{\ell -1}+\frac{1}{4}10^{\ell -1}+\frac{1}{8}9^{\ell -1}+\frac{5}{6}8^{\ell -1}+\frac{1}{6}7^{\ell -1}\\
& & 
{}+\frac{1}{4}6^{\ell -1}-\frac{13}{12}4^{\ell -1}-\frac{1}{3}3^{\ell -1}+\frac{1}{6}2^{\ell -1}+\frac{1}{24}
.
\end{eqnarray*}
\end{lemma}

\begin{proof}
We have $g_{\ell }\left( 1\right) =1$,
$g_{\ell }\left( 2\right) =2^{\ell }-1$.
Therefore,
$$N_{\ell }\left( 2\right) =\frac{1}{1!}\frac{g_{\ell }\left( 2\right) }{2}+\frac{1}{2!}\frac{\left( g_{\ell }\left( 1\right) \right) ^{2}}{1^{2}}=\frac{2^{\ell }-1}{2}+\frac{1}{2}=2^{\ell -1}.$$
As
$g_{\ell }\left( 3\right) =\frac{3^{\ell }-1}{2}$
we obtain
\begin{eqnarray*}
N_{\ell }\left( 3\right) & = & \frac{1}{1!}\frac{g_{\ell }\left( 3\right) }{3}+\frac{1}{1!1
!
}\frac{g_{\ell }\left( 1\right) g_{\ell }\left( 2\right) }{2}+\frac{1}{
3!}\frac{\left( g_{\ell }\left( 1\right) \right) ^{3}}{1^{3}}\\ &= &
\frac{3^{\ell }-1}{6}+
\frac{2^{\ell }-1}{2}+\frac{1}{6}=\frac{1}{2}3^{\ell -1}+2^{\ell -1}-\frac{1}{2}. \end{eqnarray*}

We have
$g_{\ell }\left( 4
\right) =\frac{\left( 2^{\ell +1}-1\right) \left( 2^{\ell }-1\right) }{3}=\frac{2\cdot 4^{\ell }-3\cdot 2^{\ell }+1}{3}$.
Therefore,
\begin{eqnarray*}
N_{\ell }\left( 4\right) & = &\frac{g_{\ell }\left( 4\right) }{4}+
\frac{g_{\ell }\left( 1\right) g_{\ell }\left( 3\right) }{3}+\frac{1}{2!}\frac{\left( g_{\ell }\left( 2\right) \right) ^{2}}{2^{2}}+\frac{1}{1!2!}\frac{\left( g_{\ell }\left( 1\right) \right) ^{2}g_{\ell }\left( 2\right) }{2}+\frac{1}{4!}\frac{\left( g_{\ell }\left( 1\right) \right) ^{4}}{1^{4}}\\ &= &\frac{2\cdot 4^{\ell }-3\cdot 2^{\ell }+1}{12}+\frac{3^{\ell }-1}{6}+\frac{4^{\ell }-2^{\ell +1}+1}{8}+\frac{
2^{\ell }-1}{4}+\frac{1}{24}\\ & = & \frac{7}{6}4^{\ell -1}+\frac{1}{2}3^{\ell -1}-\frac{1}{2}2^{\ell -1}-\frac{1}{6
}.
\end{eqnarray*}

We have
$g_{\ell }\left( 5\right) =\frac{5^{\ell }-1}{4}$
and we obtain
\begin{eqnarray*}
N_{\ell }\left( 5\right) &=&\frac{1}{1!}\frac{
g_{\ell }\left( 5\right)
}{5
}+\frac{1}{1!1!}\frac{g_{\ell }\left( 4\right) g_{\ell }\left( 1\right) }{4}+\frac{1}{1!1!}\frac{g_{\ell }\left( 3\right) g_{\ell }\left( 2\right) }{6}+\frac{1}{1!2!}\frac{g_{\ell }\left( 3\right) \left( g_{\ell }\left( 1\right) 
\right) ^{2}}{3}\\
& & {}+\frac{1}{2!1!}\frac{\left( g_{\ell }\left( 2\right) \right) ^{2}g_{\ell }\left( 1\right) }{2^{2}}
+\frac{1}{1!3!}\frac{g_{\ell }\left( 2\right) \left( g_{\ell }\left( 1\right) \right) ^{3}}{2}+\frac{1}{5!}\frac{\left( g_{\ell }\left( 1\right) \right) ^{5}}{1^{5}}\\
&=&\frac{5^{\ell }-1}{20}
 +\frac{2\cdot 4^{\ell }-3\cdot 2^{\ell }+1}{12}
 +\frac{6^{\ell }-3^{\ell }-2^{\ell }+1}{12}\\
& & {}+\frac{3^{\ell }-1}{6}+\frac{4^{\ell }-2^{\ell +1}+1}{8}+\frac{2^{\ell }-1}{12}+\frac{1}{120}\\
&=&\frac{1}{2}6^{\ell -1}+\frac{1}{4}5^{\ell -1}+\frac{7
}{6
}4^{\ell -1}-
2^{\ell -1}+\frac{1}{
12}.
\end{eqnarray*}
The formula
for
$N_{\ell}(7)$
is proven in the same way.
\end{proof}

We define
\begin{eqnarray*}
R_{4}\left( \ell \right) &=&N_{\ell }\left( 4\right) -\frac{7}{6}4^{\ell -1}-\frac{1}{2}3^{\ell -1},\\
R_{5}\left( \ell \right) &=&N_{\ell }\left( 5\right) -\frac{1}{2}6^{\ell-1}-\frac{1}{4}5^{\ell -1},\\
R_{7}\left( \ell \right) &=&N_{\ell }\left( 7\right) -\frac{7}{12}12^{\ell -1}-\frac{1}{4}10^{\ell -1}.
\end{eqnarray*}
Then
$R_{4}\left( \ell \right)
=
O\left( 2^{\ell }\right) $,
$R_{5}\left( \ell \right) =O\left( 4^{\ell }\right) $,
and
$R_{7}\left( \ell \right) =O\left( 9^{\ell }\right) $
using Landau's notation.
 The following application illustrates
some general pattern.
\begin{corollary}\label{Beispiele}
We have $\Delta_{2,3}^{\ell}<0$ and 
$\Delta_{3,4}^{\ell} <0$ for almost all $\ell$.
\end{corollary}
\begin{proof} We have
\begin{eqnarray*}
\Delta_{3,2}^{\ell}& = &\left( \frac{1}{2}3^{\ell -1}+2^{\ell -1}-\frac{1}{2}\right) 2^{\ell -1}-\frac{1}{2}6^{\ell -1}-\frac{1}{4}5^{\ell -1}-R_{5}\left( \ell \right)\\ & = &-\frac{1}{4}5^{\ell -1}+O\left( 4^{\ell }\right) <0
\end{eqnarray*}
for sufficiently large $\ell $. Similiary, we have
\begin{eqnarray*}
\Delta_{3,4}^{\ell} &=&\left( \frac{1}{2}3^{\ell -1}+2^{\ell -1}-\frac{1}{2}\right) \left( \frac{7}{6}4^{\ell -1}+\frac{1}{2}3^{\ell -1}+R_{4}\left( \ell \right) \right) -
\frac{7}{12}12^{\ell -1}-\frac{1}{4}10^{\ell -1}-R_{7}\left( \ell \right)\\ & =&-\frac{1}{4}10^{\ell -1}+O\left( 9^{\ell }\right) <0
\end{eqnarray*}
for 
sufficiently large $\ell$.
\end{proof}

Finally, we state the following.
Let $\ell \geq 2$. Then we have the upper and lower bounds
\begin{equation*}\label{bounds}
n^{\ell -1} \leq g_{\ell}(n) \leq n^{\ell} \leq \sigma(n) \, n^{\ell -1}.
\end{equation*}
This follows from $g_{\ell}(n)$ multiplicative and formula (\ref{local}).
\section{Proof of Theorem \ref{th:unbounded}}
Let $a,b \geq 2$. Let $n \geq 2$, then we have from (\ref{eq:asymp}) for 
$\ell$ large that $N_{\ell}(n) \sim B_{\ell}(n)$.
Note that $\Delta_{a,b}^{\ell} = \Delta_{b,a}^{\ell}$. Our strategy is to go through each case $a \pmod{3}$.
The pairs $(a,b)$ with $a \equiv 0 \pmod{3}$ and $b \in \{2,4\}$ have to treated separately. 

\subsection{
Case $a \equiv 0 \pmod{3}$}

\subsubsection{
Case $b \equiv 0 \pmod{3}$}

\[ 
\frac{N_{\ell }\left( a\right) N_{\ell }\left( b\right) }{N_{\ell }\left( a+b\right) }
\sim \frac{\frac{1}{2^{a/3}\left( \frac{a}{3}\right) ! }3^{\left( \ell -1\right) a/3}\frac{1}{2^{b/3}\left( \frac{b}{3}\right) !}3^{\left( \ell -1\right) b/3}}{\frac{1}{2^{\left( a+b\right) /3}\left( \frac{a+b}{3}\right) !}3^{\left( \ell -1\right) \left( a+b\right) /3}}=\frac{\left( \frac{a+b}{3}\right) !}{\left( \frac{a}{3}\right) !\left( \frac{b}{3}\right) !}=\binom{\frac{a+b}{3}}{\frac{a}{3}}>1
.
\]

\subsubsection{
Case $b \equiv 1 \pmod{3}$}

Then for $b \neq 4$ we have
\[
\frac{N_{\ell }\left( a\right) N_{\ell }\left( b\right) }{N_{\ell }\left( a+b\right) }\sim \frac{\frac{1}{2^{a/3}\left( \frac{a}{3}\right) !}3^{\left( \ell -1\right) a/3}\frac{7}{6}\frac{1}{2^{\frac{b-4}{3}}
\left( \frac{b-4}{3}\right) !}\left( 4\cdot 3^{\frac{b-4}{3}}\right) ^{\ell -1}}{\frac{7}{6}\frac{1}{2^{\frac{a+b-4}{3}}
\left( \frac{a+b-4}{3}\right) !}\left( 4\cdot 3^{\frac{a+b-4}{3}}\right) ^{\ell -1}}=\binom{\frac{a+b-4}{3}}{\frac{a}{3}}>1
.
\]

\subsubsection{
Case $b \equiv 2 \pmod{3}$}

Then for $b \neq 2$ we have
\[
\frac{N_{\ell }\left( a\right) N_{\ell }\left( b\right) }{N_{\ell }\left( a+b\right) }\sim \frac{\frac{1}{2^{a/3}\left( \frac{a}{3}\right) !}3^{\left( \ell -1\right) a/3}\frac{1}{2^{\frac{b-2}{3}}\left( \frac{b-2}{3}\right) !}\left( 2\cdot 3^{\frac{b-2}{3}}\right) ^{\ell -1}}{\frac{1}{2^{\frac{a+b-2}{3}}\left( \frac{a+b-2}{3}\right) !}\left( 2\cdot 3^{\frac{a+b-2}{3}}\right) ^{\ell -1}}=\binom{\frac{a+b-2}{3}}{\frac{a}{3}}>1
.
\]

\subsection{
Case $a \equiv 1 \pmod{3}$}

\subsubsection{
Case $b \equiv 1 \pmod{3}$}

\begin{eqnarray*}
\frac{N_{\ell }\left( a\right) N_{\ell }\left( b\right) }{N_{\ell }\left( a+b\right) }&\sim &\frac{\frac{7}{6}\frac{1}{2^{\frac{a-4}{3}}
\left( \frac{a-4}{3}\right) !}\left( 4\cdot 3^{\frac{a-4}{3}}\right) ^{\ell -1}\frac{7}{6}\frac{1}{2^{\frac{b-4}{3}}
\left( \frac{b-4}{3}\right) !}\left( 4\cdot 3^{\frac{b-4}{3}}\right) ^{\ell -1}}{\frac{1}{2^{\frac{a+b-2}{3}}\left( \frac{a+b-2}{3}\right) !}\left( 2\cdot 3^{\frac{a+b-2}{3}}\right) ^{\ell -1}}\\
&=&\frac{49}{2^{-2}\cdot 36}\frac{\left( \frac{a+b-2}{3}\right) !}{\left( \frac{a-4}{3}\right) !\left( \frac{b-4}{3}\right) !}\left( 8\cdot 3^{-2}\right) ^{\ell -1}\rightarrow 0.
\end{eqnarray*}
\subsubsection{
Case $b \equiv 2 \pmod{3}$}

\begin{eqnarray*}
\frac{N_{\ell }\left( a\right) N_{\ell }\left( b\right) }{N_{\ell }\left( a+b\right) }&\sim &\frac{\frac{7}{6}\frac{1}{2^{\frac{a-4}{3}}
\left( \frac{a-4}{3}\right) !}\left( 4\cdot 3^{\frac{a-4}{3}}\right) ^{\ell -1}\cdot \frac{1}{2^{\frac{b-2}{3}}\left( \frac{b-2}{3}\right) !}\left( 2\cdot 3^{\frac{b-2}{3}}\right) ^{\ell -1}}{\frac{1}{2^{\frac{a+b}{3}}\left( \frac{a+b}{3}\right) !}\left( 3^{\left( a+b\right) /3}\right) ^{\ell -1}}\\
&=&\frac{7}{6\cdot 2^{-2}}\frac{\left( \frac{a+b}{3}\right) !}{\left( \frac{a-4}{3}\right) !\left( \frac{b-2}{3}\right) !}\left ( \frac{8}{9}\right) ^{\ell -1}\rightarrow 0
.
\end{eqnarray*}

\subsection{
Case $a \equiv 2 \pmod{3}$}

\subsubsection{
Case $b \equiv 2 \pmod{3}$}

\[
\frac{N_{\ell }\left( a\right) N_{\ell }\left( b\right) }{N_{\ell }\left( a+b\right) }\sim \frac{\frac{1}{2^{\frac{a-2}{3}}\left( \frac{a-2}{3}\right) !}\left( 2\cdot 3^{\frac{a-2}{3}}\right) ^{\ell -1}\frac{1}{2^{\frac{b-2}{3}}\left( \frac{b-2}{3}\right) !}\left( 2\cdot 3^{\frac{b-2}{3}}\right) ^{\ell -1}}{\frac{7}{6}\frac{1}{2^{\frac{a+b-4}{3}}
\left( \frac{a+b-4}{3}\right) !}\left( 4\cdot 3^{\frac{a+b-4}{3}}\right) ^{\ell -1}}=\frac {6}{7}\binom{\frac{a+b-4}{3}}{\frac{a-2}{3}}
.
\]
For $a,b >2$ we conclude that $\Delta_{a,b}^{\ell} >0$ for $\ell$ large.
Otherwise $\Delta_{a,b}^{\ell} <0$ for $a$ or $b$ equal to $2$ and $\ell$ large (and $a,b \equiv 2 \pmod{3}$).
%
\subsection{The cases $a \equiv 0 \pmod{3}$ and $b = 2,4$}
The cases $\Delta_{3,2}^{\ell}$ and $\Delta_{3,4}^{\ell}$ had been already treated in Corollary \ref{Beispiele}.
For the general cases $a>3$:  $(a,2)$ and $(a,4)$ with $a \equiv 0 \pmod{3}$,
the following property is useful.


\begin{lemma}
\label{zweite}Let
\[
D_{\ell }\left( n\right) =\left\{
\begin{array}{ll}
\frac{
5}{6
}\frac{1}{ 2^{\frac{n-6}{3}}\left( \frac{n-6}{3}\right) !}\left( 2^{3}\cdot 3^{\frac{n-6}{3}}\right) ^{\ell -1}
,&n\equiv 0
\mod 3
,
n>3,
\\
\frac{41}{120}\frac{1}{2^{\frac{n-10}{3}}\left( \frac{n-10}{3}\right)!}\left( 2^{5}\cdot 3^{\frac{n-10}{3}}\right) ^{\ell -1}
,&n\equiv 1
\mod 3
,
n>7,
\\
\frac{43}{72}\frac{1}{2^{\frac{n-8}{3}}\left( \frac{n-8}{3}\right) !}\left( 2^{4}\cdot 3^{\frac{n-8}{3}}\right) ^{\ell -1}
,&n\equiv 2\mod 3,n>5.
\end{array}
\right.
\]
Then
we have
$
N_{\ell }\left( n
\right) =B
_{\ell }\left( n\right) +\tilde{B}_{\ell }\left( n\right) $
with
\[
\tilde{B}_{\ell }\left( n\right) =D_{\ell }\left( n\right) +
\left\{
\begin{array}{ll}
o\left( \left( 2^{3}\cdot 3^{\frac{n-6}{3}}\right) ^{\ell }\right)
,&n\equiv 0
\mod 3
,
n>3,
\\
o\left( \left( 2^{5}\cdot 3^{\frac{n-10}{3}}\right) ^{\ell }\right)
,&n\equiv 1
\mod 3
,
n>7,
\\
o\left( \left( 2^{4}\cdot 3^{\frac{n-8}{3}}\right) ^{\ell }\right)
,&n\equiv 2\mod 3,n>5.
\end{array}
\right.
\]
\end{lemma}

Let now $a\equiv 0\mod 3$, $a>3$. Then
\begin{eqnarray*}
&&{N_{\ell }\left( a\right) N_{\ell }\left( 4\right) }-{N_{\ell }\left( a+4\right) }\\
&=&
\left( B
_{\ell }\left( a
\right) +\tilde{B}_{\ell }\left(
a\right)
\right) \left( \frac{7}{6}4^{\ell -1}+\frac{1}{2}3^{\ell -1}+O\left( 2^{\ell }\right) \right) -
B
_{\ell }\left( a+4\right) -\tilde{B}_{\ell }\left( a+4\right) \\
&=&\left( \frac{35}{36}-\frac{41}{120}\right) \frac{1}{2^{\frac{a-6}{3}}\left( \frac{a-6}{3}\right) !}
\left ( 2^{5}\cdot 3^{\frac{a-6}{3}}\right) ^{\ell -1}+o\left( \left( 2^{5}\cdot 3^{\frac{a-6}{3}}\right) ^{\ell }\right) >0
.
\end{eqnarray*}

Let $a\equiv 0
\mod 3$. Then
\begin{eqnarray*}
&&{N_{\ell }\left( a\right) N_{\ell }\left( 2\right) }-{N_{\ell}\left( a+2\right) }\\
&=&
\left( B
_{\ell }\left( a\right) +\tilde{B}_{\ell }\left( a\right)
\right) 2^{\ell -1}-
B
_{\ell }\left( a+2\right) -\tilde{B}_{\ell }\left( a+2\right) \\
&=&\left( \frac{5}{6}-\frac{43}{72}\right) \frac{1}{2^{\frac{a-6}{3}}\cdot \left( \frac{a-6}{3}\right) }\left( 2^{4}\cdot 3^{\frac{a-6}{3}}\right) ^{\ell -1}+o\left( \left( 2^{4}\cdot 3^{\frac{a-6}{3}}\right) ^{\ell }\right) >0
.
\end{eqnarray*}
\section{Proof of Theorem \ref{th:explicit 1}
and Theorem \ref{th:explicit 2}}
Let $n \geq 2$. As proposed in \cite{AHN24} let
\[
M_{1}\left( n\right) =\left\{
\begin{array}{ll}
3^{n/3},&n\equiv 0\mod 3,\\
4\cdot 3^{\frac{n-4}{3}},&n\equiv 1\mod 3,\\
2\cdot 3^{\frac{n-2}{3}},&n\equiv 2\mod 3,
\end{array}
\right.
\]
and
\[
A_{\ell }\left( n\right) =\left\{
\begin{array}{ll}
\frac{1}{\left( \frac{n}{3}\right) !}3^{\left( \ell -2\right) n/3},&n\equiv 0\mod 3,\\
\frac{2}{3}\frac{1}{\left( \frac{n-4}{3}\right) !}\left( 4\cdot 3^{\frac{n-4}{3}}\right) ^{\ell -2},&n\equiv 1\mod 3,\\
\frac{1}{\left( \frac{n-2}{3}\right) !}\left( 2\cdot 3^{\frac{n-2}{3}}\right) ^{\ell -2}, &n\equiv 2\mod 3.
\end{array}
\right.
\]
Let $\ell \geq 2$ and $p(n)=N_{2}(n)$ the number of partitions of $n$, then
\begin{equation} \label{A: bounds}
A_{\ell }\left( n\right) \leq N_{\ell }\left( n\right) \leq p\left( n\right) \left( M_{1}\left( n\right) \right) ^{\ell -1}.
\end{equation}
It turns out that these lower and upper bounds of $N_{\ell}(n)$ are suffient to
prove most of the parts of Theorem \ref{th:explicit 1}
and Theorem \ref{th:explicit 2}.
We follow the same structure of the proof as of the proof of Theorem \ref{eq:asymp} to demonstrate the similarities and differences.
In the following we estimate $p(n)$ by
$2^{n}$.

\subsection{
Case $a \equiv 0 \pmod{3}$}

We consider the three cases $b \equiv 0,1,2 \pmod{3}$, where $b=2$ and $b=4$ have to be considered separately.

\subsubsection{
Case $b \equiv 0 \pmod{3}$} 

Let $\ell >1+\frac{a+b}{3}\log _{9/8}\left( 4\frac{a+b}{3}\right) $. Then
\begin{eqnarray*}
\frac{N_{\ell }\left( a\right) N_{\ell }\left( b\right) }{N_{\ell }\left( a+b\right) }&\geq &\frac{\frac{1}{2^{\frac{a}{3}}\left( \frac{a}{3}\right) !}\frac{1}{2^{\frac{b}{3}}\left( \frac{b}{3}\right) !}3^{\left( \ell -1\right) \frac{a+b}{3}}}{\frac{1}{2^{\frac{a+b}{3}}\left( \frac{a+b}{3}\right) !}3^{\left( \ell -1\right) \frac{a+b}{3}}+\left( 2^{3}\cdot 3^{\frac{a+b-6}{3}}\right) ^{\ell -1}p\left( a+b\right) }\\
&=&\binom{\frac{a+b}{3}}{\frac{a}{3}}\frac{1}{1+2^{\frac{a+b}{3}}\left( \frac{a+b}{3}\right) !p\left( a+b\right) \left( \frac{8}{9}\right) ^{\ell -1}}\\
&\geq &\binom{\frac{a+b}{3}}{\frac{a}{3}}\frac{1}{1+\left( 4
\frac{a+b}{3}\right) ^{\frac{a+b}{3}}\left( \frac{8}{9}\right) ^{\ell -1}}>1
\end{eqnarray*}

\subsubsection{
Case $b \equiv 1 \pmod{3}$, $b\neq 4$} 

Let 
$$\ell \geq \log _{9/8}\left( \frac{7}{6}\right) +\frac{a+b-2}{3}\log _{9/8}\left( 4\right) +\frac{a+b-4}{3}\log {9/8}\left( \frac{a+b-4}{3}\right),$$
then the following quotient is larger than $1$, and therefore, the Bessenrodt--Ono inequality is satisfied:
\begin{eqnarray*}
\frac{N_{\ell }\left( a\right) N_{\ell }\left( b\right) }{N_{\ell }\left( a+b\right) }
&\geq& \frac{\frac{7}{6}\frac{1}{2^{\frac{a}{3}}\left( \frac{a}{3}\right) !}\frac{1}{2^{\frac{b-4}{3}}\left( \frac{b-4}{3}\right) !}\left( 4\cdot 3^{\frac{a+b-4}{3}}\right) ^{\ell -1}}{\frac{7}{6}\frac{1}{2^{\frac{a+b-4}{3}}\left( \frac{a+b-4}{3}\right) !}\left( 4\cdot 3^{\frac{a+b-4}{3}}\right) ^{\ell -1}+p\left( a+b\right) \left( 2^{5}\cdot 3^{\frac{a+b-10}{3}}\right) ^{\ell -1}}\\
&=& \binom{\frac{a+b-4}{3}}{\frac{a}{3}}\frac{1}{1+\frac{6}{7}2^{a+b-4}\left( \frac{a+b-4}{3}\right) !p\left( a+b\right) \left( \frac{8}{9}\right) ^{\ell -1}}\\
&>&\binom{\frac{a+b-4}{3}}{\frac{a}{3}}\frac{1}{1+\frac{6}{7}4^{\frac{a+b-2}{3}}\left( \frac{a+b-4}{3}\right) ^{\frac{a+b-4}{3}}\left( \frac{8}{9}\right) ^{\ell -1}}.
\end{eqnarray*}

\subsubsection{
Case $b\equiv 2 \pmod{3}$ and $b \neq 2$}

Let
$$\ell \geq 1+\frac{a+b-1}{3}\log _{9/8}\left( 4\right) +\frac{a+b-2}{3}\log _{9/8}\left( \frac{a+b-2}{3}\right). $$
Then we have
$\frac{N_{\ell }\left( a\right)
N_{\ell }\left( b\right) }{N_{\ell }\left( a+b\right) }
>\binom{\frac{a+b-2}{3}}{\frac{a}{3}}\frac{1}{1+4^{\frac{a+b-1}{3}}\left( \frac{a+b-2}{3}\right) ^{\frac{a+b-2}{3}}\left( \frac{8}{9}\right) ^{\ell -1}}\geq 1$.

\subsection{
Case $a \equiv 1 \pmod{3}$}

The proofs for $b\equiv 1,2 \pmod{3}$ are straightforward.

\subsubsection{
Case $b \equiv 1 \pmod{3}$}

Let
\begin{eqnarray*}
\ell
&\geq&
-6+9\left( a+b\right)
+\frac{\left( a+b\right)
/2+2}{3}\log _{9/8}\left( \frac{
a+b-2}{3}\right) \\
&>&
2+\left(
a+b\right)
\log _{9/8}\left( 2\right) +
\log _{9/8}\left( 4/9\right)
+\frac{
a+b
/2+2}{3}\log _{9/8}\left( \frac{
a+b
-2}{3}\right)\\
& &+\log _{9/8}\left( 16\right)
+\frac{
a+b
-8}{3}\log _{9/8}\left( 3\right).
\end{eqnarray*}
Then
\begin{eqnarray*}
\frac{N_{\ell }\left( a\right) N_{\ell }\left( b\right) }{N_{\ell }\left(
a+b
\right) }
&\leq &p\left( a\right) p\left( b\right) \frac{4}{9}\frac{\left( \frac{a+b
-2}{3}\right) !}{\left( \frac{a-4}{3}\right) !\left( \frac{b-4}{3}\right) !}16\cdot 3^{\frac{
-8}{3}}\left( \frac{8}{9}\right) ^{\ell-2}\\
&\leq &\frac{2^{
\left( a+b\right)
+6}}{9}\left( \frac{a+b
-2}{3}\right) ^{\frac{\left( a+b\right)
/2+2}{3}}
3^{\frac{
a+b
-8}{3}}
\left( \frac{8}{9}\right) ^{\ell -2}
<1.
\end{eqnarray*}

\subsubsection{
Case $b \equiv 2 \pmod{3}$}

\begin{eqnarray*}
\frac{N_{\ell }\left( a\right) N_{\ell }\left( b\right) }{N_{\ell }\left( a+b\right) }&\leq &\frac{p\left( a\right) p\left( b\right) \frac{3}{2}\frac{1}{\left( \frac{a-4}{3}\right) !}\frac{1}{\left( \frac{b-2}{3}\right) !}\left( 8\cdot 3^{\frac{a+b-6}{3}}\right) ^{\ell -1}}{\frac{1}{\left( \frac{a+b}{3}\right) !}3^{\left( \ell -2\right) \left( a+b\right) /3}}\\
&\leq &\frac{3}{2}2^{n}\left( \frac{n}{3}\right) ^{n/2+4}\left( 8\cdot 3^{\frac{n-6}{3}}\right) \left( \frac{8}{9}\right) ^{\ell -2}<1
\end{eqnarray*}
if
\begin{eqnarray*}
\ell
&\geq&
5+9\left( a+b\right)
+\left( \frac{a+b
}{2}+4\right) \log _{9/8}\left( \frac{
a+b}{3}\right)\\
&>&
2+\log _{9/8}\left( \frac{3}{2}\right) +n\log _{9/8}\left( 2\right) +
\frac{\left( a+b\right)
/2+4}{3}
\log _{9/8}\left( \frac{a+b
}{3}\right) \\ & &+\log  _{9/8}\left( 8\right) +\frac{a+b
-6}{3}\log _{9/8}\left( 3\right).
\end{eqnarray*}
\subsection{
Case $a \equiv 2 \pmod{3}$ and $b\equiv 2 \pmod{3}$} 

\subsubsection{
Case $a,b>2$}

Let
$$\ell \geq 1+\log _{9/8}\left( \frac{7}{6}\right) +\left( \frac{a+b-4}{3}+a+b\right) \log _{9/8}\left( 2\right).$$
Then we have the following inequality
\[
\frac{N_{\ell }\left( a\right) N_{\ell
}\left( b\right) }{N_{\ell }\left( a+b\right) }
>
\frac{6}{7}\binom{\frac{a+b-4}{3}}
{\frac{a-2}{3}}\frac{1}{1+\frac{6}{7}
2^{\frac{a+b-4}{3}}2^{a+b}\left( \frac{8}{9}\right) ^{\ell -1}}\geq 1.
\]

\subsubsection{
Case $a \equiv 2 \pmod{3}$ and $b=2$}

We start with the case $(a,b)=(2,2)$. \\
We obtain from (\ref{explicit}) that for $\ell >1$:
\begin{eqnarray*}
\frac{N_{\ell }\left( 4\right) }{\left( N_{\ell }
\left( 2\right) \right) ^{2}}
&=&\frac{7}{6}+\frac{1}{2}\left( \frac{3}{4}\right) ^{\ell -1}-\frac{1}{2}2^{1-\ell }-\frac{1}{2}4^{1-\ell }\\
&=& \frac{7}{6}+\frac{3}{8}\left( \frac{3}{4}\right) ^{\ell -2}-\frac{1}{4}2^{2-\ell }-\frac{1}{8}4^{2-\ell }
\geq \frac{7}{6}>1.
\end{eqnarray*}
Further, let $(a,b)= (5,2)$ and let
$\ell \geq 16>1+\log _{6/5}\left( 6\frac{17}{7}\right) $. Then
\begin{eqnarray*}
\frac{N_{\ell }\left( 5\right) N_{\ell }\left( 2\right) }{N_{\ell }\left( 7\right) } & < & \frac{\frac{1}{2}12^{\ell -1}+\frac{17}{12}10^{\ell -1}}{\frac{7}{12}12^{\ell -1}}\\
&=& \frac{6}{7}+\frac{17}{7}\left( \frac{5}{6}\right) ^{\ell -1}\leq 1.
\end{eqnarray*}
Actually, $\Delta_{5,2}^{\ell} <1$ for all $\ell$. 
The remaining cases $1\leq \ell \leq 15$ can be checked directly.
Finally, let $a\equiv 2
\pmod{3}$ and $a>5$. Then we have
\begin{eqnarray*}
N_{\ell }\left( a\right) & \leq & \frac{1}{2^{\frac{a-2}{3}}\left( \frac{a-2}{3}\right) !}\left( 2\cdot 3^{\frac{a-2}{3}}\right) ^{\ell -1}+p\left( a\right) \left( 2^{4}\cdot 3^{\frac{a-8}{3}}\right) ^{\ell -1} \\
N_{\ell }\left( a+2\right) & \geq & \frac{7}{6}\frac{1}{2^{\frac{a-2}{3}}\left( \frac{a-2}{3}\right) !}\left( 4\cdot 3^{\frac{a-2}{3}}\right) ^{\ell -1}.
\end{eqnarray*}
Therefore, we obtain for 
$\ell \geq 1+\log _{9/8}\left( 6\right) +\frac{a-2}{3}\log _{9/8}\left( 2\frac{a-2}{3}\right) +a\log _{9/8}\left(
2\right) $ that 
\begin{eqnarray*}
\frac{N_{\ell }\left( a\right) N_{\ell }\left( 2\right) }{N_{\ell }\left( a+2\right) }
&<&\frac{6}{7}\left(
1+2^{\frac{a-2}{3}}\left( \frac{a-2}{3}\right) !p\left( a\right) \left( \frac{8}{9}\right) ^{\ell -1}\right) \\
&<&\frac{6}{7}\left( 1+\left( 2\frac{a-2}{3}\right) ^{\frac{a-2}{3}}2^{a}\left( \frac{8}{9}\right) ^{\ell -1}
\right) \leq 1.
\end{eqnarray*}
We still have to deal with $a \equiv 0 \pmod{3}$ and $b =2,4$.
It turns out that these are the most difficult cases.

\subsection{
Case $a \equiv 0 \pmod{3}$ and $b=2,4$}

In this case we first study $a=3,6,9$ explictly.

\subsubsection{
Case $(a,b)=(3,2)$}

Let $R_{5}\left( \ell \right) >\frac{1}{6}4^{\ell -1}$. 
If
$\ell \geq 7
>1+\log _{5/4}\left( 10/3\right) $ 
then
\begin{eqnarray*}
N_{\ell }\left( 2\right) N_{\ell }\left( 3\right) -N_{\ell }\left( 5\right)
&=& \left( \frac{1}{2}3^{\ell -1}+2^{\ell -1}-\frac{1}{2}\right) 2^{\ell -1}-\frac{1}{2}6^{\ell -1}-\frac{1}{4}5^{\ell-1}-R_{5}\left( \ell \right) \\
&=&-\frac{1}{4}5^{\ell -1}+4^{\ell -1}-\frac{1}{2}2^{\ell -1}-R_{5}\left( \ell \right)\\
&<
&-\frac{1}{4}5^{\ell -1}+\frac{5}{
6}4^{\ell -1}\leq 0.
\end{eqnarray*}
Moreover, for all $\ell \geq 2$ we have $\Delta_{3,2}^{\ell} <0$. Here
the remaining cases $1\leq \ell \leq 6$ can be
checked directly.

\subsubsection{
Case $(a,b)=(3,4)$}

We show $\Delta_{a,3}^{\ell}<0$ for all $\ell \geq 2$.
Since $N_{\ell }\left( 3\right) <\frac{1}{2}3^{\ell -1}+2^{\ell -1}$
and $R_{7}\left( \ell \right) \geq -\frac{1}{24}3^{\ell -1}$
(see Lemma \ref{klein}):
\begin{eqnarray*}
N_{\ell }\left( 3\right) N_{\ell }\left( 4\right) -N_{\ell }\left( 7\right)
&<&
\left( \frac{1}{2}3^{\ell -1}
+2^{\ell -1}
\right) \left( \frac{7}{6}4^{\ell -1}+\frac{1}{2}3^{\ell -1}\right)\\
& & {}-\frac{7}{12}12^{\ell -1}-\frac{1}{4}10^{\ell -1}
-R_{7}\left( \ell \right) \\
&\leq & -\frac{1}{4}10^{\ell -1}+\frac{1}{4}9^{\ell -1}
+\frac{7}{6}8^{\ell -1}+\frac{1}{2}6^{\ell -1}+\frac{1}{24}3^{\ell -1}
\\&\leq & -\frac{1}{4}10^{\ell -1}+\frac{47}{24}9^{\ell -1}\leq 0.
\end{eqnarray*}
if
$\ell \geq 20>
\log _{10/9}\left( \frac{47}{6}\right) $. We
checked the
remaining cases $2\leq \ell \leq 19$
 with
PARI/GP.

Before we continue, we record the 
next-to-subleading term.

\subsubsection{Growth expansion}

In the following we need even the 
next-to-subleading term so that we can
determine whether the Bessenrodt--Ono
inequality holds or not. For completeness
we determine all possible cases. Its proof
is similar that of the 
leading or 
subleading term.

\begin{lemma}
\label{dritte}For the 
next-to-subleading terms we
obtain
\[
M_{3}\left( n\right) =\left\{
\begin{array}{ll}
1,&n=3,\\
2,&n=4,\\
4,&n=5,\\
6,&n=6,\\
9,&n=7,\\
15,&n=8,\\
20,&n=9,\\
30,&n=10,\\
45,&n=11,\\
90,&n=13,\\
2^{6}\cdot 3^{\frac{n-12}{3}},&n\equiv 0\mod 3,n>9,\\
2^{8}\cdot 3^{\frac{n-16}{3}},&n\equiv 1\mod 3,n>13,\\
2^{7}\cdot 3^{\frac{n-14}{3}},&n\equiv 2\mod 3,n>11.
\end{array}
\right.
\]
\end{lemma}

\begin{remark}
For $n
>1
3$ we can now estimate the growth of
$N_{\ell }\left( n\right) $ in $\ell $ in the
following way:
\[
\left| N_{\ell }\left( n\right) -B_{\ell }\left( n\right) -D_{\ell }\left( n\right) \right| 
\leq
\begin{array}{ll}
p\left( n\right) \left( 2^{6}\cdot 3^{\frac{n-12}{3}}\right) ^{\ell -1}
,&n\equiv 0\pmod{3},\\
p\left( n\right) \left( 2^{8}\cdot 3^{\frac{n-16
}{3}}\right) ^{\ell -1},&n\equiv 1\pmod{3},\\
p\left( n\right) \left( 2^{7}
\cdot 3
^{\frac{n-14}{3}}\right) ^{\ell -1},
&n\equiv 2\pmod{3}.
\end{array}
\]
For $B_{\ell }\left( n\right) $ and
$D_{\ell }\left( n\right) $ see
Lemma~\ref{zweite}.
\end{remark}

\subsubsection{
Case $(a,b)=(6,2)$ and $(6,4)$}

Using Lemma~\ref{dritte}:
\begin{eqnarray*}
N_{\ell }\left( 6\right) &\geq& \frac{1}{8}9^{\ell -1}
+\frac{5}{6}8^{\ell -1},\\
N_{\ell }\left( 8\right) &\leq& \frac{1}{48}18^{\ell -1}+\frac{43}{72}16^{\ell -1}+22\cdot 15^{\ell -1}.
\end{eqnarray*}
This leads to
\begin{eqnarray*}
&&N_{\ell }\left( 6\right) N_{\ell }\left( 2\right) -N_{\ell }\left( 8\right) \\
&\geq& 
\left( \frac{1}{48}9^{\ell -1}+\frac{5}{6}8^{\ell -1}\right) 2^{\ell -1}-\frac{1}{48}18^{\ell -1}
-\frac{43}{72}16^{\ell -1}-22\cdot 15^{\ell -1}\\
&=&\frac{17}{72}16^{\ell -1}-22\cdot 15^{\ell -1}>0.
\end{eqnarray*}
if
$\ell \geq 72>1+\log _{16/15}\left( 22\frac{72}{17}\right) $. We checked
that also in the cases
$3\leq \ell \leq 71$
this is positive. In the case $\ell =2$ this
$=0$ as is known from \cite{BO16}.

Let $(a,b)=(6,4)$. We prove that $\Delta_{6,4}^{\ell} >0$ for $\ell \geq 2$.
Using Lemma~\ref{zweite}
und~\ref{dritte} we can observe that
$$ N_{\ell }\left( 10\right) \leq \frac{7}{48}36^{\ell -1}+\frac{41}{120}32^{\ell -1}+42\cdot 30^{\ell -1}$$
and obtain
$N_{\ell }\left( 6\right) N_{\ell }\left( 4\right) -N_{\ell }\left( 10\right) \geq \frac{227}{360}32^{\ell -1}-42\cdot 30^{\ell -1}>0$
if
$\ell \geq 52>1+\log _{16/15}\left( 42\frac{360}{227}\right) $.
We checked the remaining cases
$2\leq \ell \leq 51$ with PARI/GP.

\subsubsection{
Case $(a,b)=(9,2)$ and $(9,4)$}

Using Lemma~\ref{dritte} we observe that
\begin{eqnarray*}
N_{\ell }\left( 9\right) &\geq& \frac{1}{48}27^{\ell -1}+\frac{5}{12}24^{\ell -1},\\
N\left( 11\right) &\leq& \frac{1}{48}54^{\ell -1}+\frac{43}{144}48^{\ell -1}+56\cdot 45^{\ell -1}
\end{eqnarray*}
and obtain
$$N_{\ell }\left( 9\right) N_{\ell }\left( 2\right) -N_{\ell }\left( 11\right) \geq \frac{17}{144}48^{\ell -1}-56\cdot 45^{\ell -1}>0$$
if
$\ell \geq 97>1+\log _{16/15}\left( 144/17\right) +\log _{16/15}\left( 56\right) $.
We checked the cases $2\leq \ell \leq 96$
with PARI/GP.

Let $(a,b)=(9,4)$. 
We can observe using Lemma~\ref{dritte}
that
$N_{\ell }\left( 13\right) \leq \frac{7}{288}108^{\ell -1}+\frac{41}{240}96^{\ell -1}+101\cdot 90^{\ell -1}$
and obtain
\begin{eqnarray*}
&&N_{\ell }\left( 9\right) N_{\ell }\left( 4\right) -N_{\ell }\left( 13\right) \\
&\geq& 
\left( \frac{1}{48}27^{\ell -1}+\frac{5}{12}24^{\ell -1}\right) \frac{7}{6}4^{\ell -1}-\frac{7}{288}108^{\ell -1}
-\frac{41}{240}96^{\ell -1}-101\cdot 90^{\ell -1}\\ &\geq& \frac{227}{720}96^{\ell -1}-101\cdot 90^{\ell -1}>0. \end{eqnarray*}
if
$\ell \geq 91>1+\log _{16/15}\left( 101\frac{720}{227}\right) $.
The remaining cases $\ell <90$
are verified positive 
via PARI/GP.

\subsubsection{
Case $a>9$ and $a\equiv 0 \pmod{3}$ and $b=2,4$}
We start with $b=4$.
Using Lemma~\ref{klein} and Lemma \ref{dritte} we observe that
\begin{eqnarray*}
N_{\ell }\left( 4\right)&=&\frac{7}{6}4^{\ell -1}+\frac{3}{2}3^{\ell -2}-2^{\ell -2}-\frac{1}{2}\geq \frac{7}{6}4^{\ell -1},\\
N_{\ell }\left( a\right) &\geq& \frac{1}{2^{\frac{a}{3}}\left( \frac{a}{3}\right) !}3^{\left( \ell -1\right) a/3}+\frac{5}{6}\frac{1}{2^{\frac{a-6}{3}}\left( \frac{a-6}{3}\right) !}\left( 2^{3}\cdot 3^{\frac{a-6}{3}}\right) ^{\ell -1},\\
N_{\ell }\left( a+4\right) &\leq&  \frac{7}{6}\frac{1}{2^{\frac{a-4}{3}}\left( \frac{a-4}{3}\right) !}\left( 4\cdot 3^{\frac{a-4}{3}}\right) ^{\ell -1}+\frac{41}{120}\frac{1}{2^{\frac{a-6}{3}}\left( \frac{a-6}{3}\right) !}\left( 2^{5}\cdot 3^{\frac{a-6}{3}}\right) ^{\ell -1}\\
& & {}+ p\left( a+4\right) \left( 2^{8}\cdot 3^{\frac{a-12}{3}}\right) ^{\ell -1}.
\end{eqnarray*}
Then $N_{\ell }\left( a\right) N_{\ell }\left( 4\right) -N_{\ell }\left( a+4\right)$
\begin{eqnarray*}
&\geq &\left( \frac{1}{2^{\frac{a}{3}}\left( \frac{a}{3}\right) !}3^{\left( \ell -1\right) a/3}+\frac{5}{6}\frac{1}{2^{\frac{a-6}{3}}\left( \frac{a-6}{3}\right) !}\left( 2^{3}\cdot 3^{\frac{a-6}{3}}\right) ^{\ell -1}\right) \frac{7}{6}4^{\ell -1}-\frac{7}{6}\frac{1}{2^{\frac{a}{3}}\left( \frac{a}{3}\right) !}\left( 4\cdot 3^{\frac{a}{3}}\right) ^{\ell -1}\\
& &{}-\frac{41}{120}\frac{1}{2^{\frac{a-6}{3}}\left( \frac{a-6}{3}\right) !}\left( 2^{5}\cdot 3^{\frac{a-6}{3}}\right) ^{\ell -1}-p\left( a+4\right) \left( 2^{8}\cdot 3^{\frac{a-12}{3}}\right) ^{\ell -1}\\
&=&\frac{227}{360}\frac{1}{2^{\frac{a-6}{3}}\left( \frac{a}{3}\right) !}\left( 2^{5}\cdot 3^{\frac{a-6}{3}}\right) ^{\ell -1}-p\left( a+4\right) \left( 2^{8}\cdot 3^{\frac{a-12}{3}}\right) ^{\ell -1}\\
&>&\left( \frac{227}{360}\frac{1}{2^{\frac{a-6}{3}}\left( \frac{a}{3}\right) ^{\frac{a}{3}}}-2^{a+4}\left( \frac{8}{9}\right) ^{\ell -1}\right) \left( 2^{5}\cdot 3^{\frac{a-6}{3}}\right) ^{\ell -1}\geq 1
\end{eqnarray*}
if
$\ell \geq 1+\log _{9/8}\left( \frac{360}{227}\right) +\frac{a-6}{3}\log _{9/8}\left( 2\right) +\frac{a-6}{3}\log _{9/8}\left( \frac{a-6}{3}\right) +\left( a+4\right) \log _{9/8}\left( 2\right) $.

Finally, let $b=2$. Then by Lemma~\ref{dritte} we obtain
\begin{eqnarray*}
N_{\ell }\left( a+2\right) &\leq& \frac{1}{2^{\frac{a}{3}}\left( \frac{a}{3}\right) !}\left( 2\cdot 3^{\frac{a}{3}}\right) ^{\ell -1}+\frac{43}{72}\frac{1}{2^{\frac{a-6}{3}}\left( \frac{a-6}{3}\right) }\left( 2^{4}\cdot 3^{\frac{a-6}{3}}\right) ^{\ell -1}\\ & & +p\left( a+2\right) \left( 2^{7}\cdot 3^{\frac{a-12}{3}}\right) ^{\ell -1}.
\end{eqnarray*}
This leads to $N_{\ell }\left( a\right) N_{\ell }\left( 2\right) -N_{\ell }\left( a+2\right)$
\begin{eqnarray*}
&\geq&  
\left( \frac{1}{2^{\frac{a}{3}}\left( \frac{a}{3}\right) !}3^{\left( \ell -1\right) a/3}+\frac{5}{6}\frac{1}{2^{\frac{a-6}{3}}\left( \frac{a-6}{3}\right) !}\left( 2^{3}\cdot 3^{\frac{a-6}{3}}\right) ^{\ell -1}\right) 2^{\ell -1}
\\
& & {}-\frac{1}{2^{\frac{a}{3}}\left( \frac{a}{3}\right) !}\left( 2\cdot 3^{\frac{a}{3}}\right) ^{\ell -1}+\frac{43}{72}\frac{1}{2^{\frac{a-6}{3}}\big( \frac{a-6}{3}\big) !}\left( 2^{4}\cdot 
3^{\frac{a-6}{3}}\right)^{\ell -1}\\ 
& & {}+p \left( a+2 \right) 
\left( 2^{7}\cdot 
3^{\frac{a-12}{3}}\right) ^{\ell -1}\\ &=&\frac{17}{72}\frac{1}{2^{\frac{a-6}{3}}\left( \frac{a-6}{3}\right) !}\left( 2^{4}\cdot 3^{\frac{a-6}{3}}\right) ^{\ell -1}-p\left( a+2\right) \left( 2^{7}\cdot 
3^{\frac{a-12}{3}}\right) ^{\ell -1}\\ &>& \frac{17}{72}\frac{1}{2^{\frac{a-6}{3}}\left( \frac{a-6}{3}\right) ^{\frac{a-6}{3}}}\left( 2^{4}\cdot 3^{\frac{a-6}{3}}\right) -2^{a+2}\left( 2^{7}\cdot 3^{\frac{a-12}{3}}\right) ^{\ell -1}\geq 0
\end{eqnarray*}
if
$\ell \geq 1+\log _{9/8}\left( 72/17\right) +\frac{a-6}{3}\log _{9/8}\left( 2\right) +\frac{a-6}{3}\log _{9/8}\left( \frac{a-6}{3}\right) +\left( a+2\right) \log _{9/8}\left( 2\right) $.
\section{Proof of Theorem \ref{th:unbounded}}
It follows from Theorem \ref{th:explicit 2} b) ii)
for $a,b \equiv 1 \pmod{3}$, $b>1$, that there is a $L_{a,b}(\ell)$
such that for $\ell \geq L_{a,b}(\ell)$ we have $\Delta_{a,b}^{\ell}<0$. Let $L_{a,b}(\ell)$ be equal to
\begin{equation*}
-6+9 (a+b)+\frac{(a+b)/2+2}{3}\log _{9/8}\left( \frac{(a+b)-2}{3}\right).
\end{equation*}
Therefore, 
for given $L
$, 
there exist always $a,b >1$, $a,b \equiv 1 \pmod{3}$ such that $L_{a,b}(\ell) > L$.

\end{document}